\documentclass[12pt, a4paper, leqno]{amsart}
\usepackage[utf8]{inputenc}
\usepackage{amsmath, amsthm}
\usepackage{amsfonts}
\usepackage{amssymb}
\usepackage{times, fullpage}
\usepackage[T1]{fontenc} %skandit 
\usepackage{url} 
\usepackage{color,esint}
\usepackage[dvipsnames]{xcolor}
\usepackage[colorlinks, allcolors=RedViolet,pdfstartview=,pdfpagemode=UseNone]{hyperref} 
\usepackage{graphicx} 
\usepackage{enumitem}
\usepackage{tabularx}
\usepackage{mathtools}

\let\oldmarginpar\marginpar
\renewcommand\marginpar[1]{\-\oldmarginpar[\raggedleft\footnotesize #1]%
{\raggedright\footnotesize #1}}

\newtheorem{thm}[equation]{Theorem}
\newtheorem{lem}[equation]{Lemma}
\newtheorem{cor}[equation]{Corollary}
\newtheorem{prp}[equation]{Proposition}

\theoremstyle{definition}

\newtheorem{dfn}[equation]{Definition}

\theoremstyle{remark}
\newtheorem{example}[equation]{Example}

\newtheorem{rem}[equation]{Remark}

\numberwithin{equation}{section}

\newcommand{\ainc}[1]{\hyperref[defn:aInc]{{\normalfont(aInc){\ensuremath{_{#1}}}}}}
\newcommand{\adec}[1]{\hyperref[defn:aDec]{{\normalfont(aDec){\ensuremath{_{#1}}}}}}
\newcommand{\inc}[1]{\hyperref[defn:Inc]{{\normalfont(Inc){\ensuremath{_{#1}}}}}}
\newcommand{\dec}[1]{\hyperref[defn:Dec]{{\normalfont(Dec){\ensuremath{_{#1}}}}}}
\newcommand{\azero}{\hyperref[def:a0]{{\normalfont(A0)}}}

\newcommand{\Phiw}{\Phi_{\text{\rm w}}}

\renewcommand{\phi}{\varphi}
\renewcommand{\epsilon}{\varepsilon}
\renewcommand{\rho}{\varrho}
\def\le{\leqslant}

\def\ge{\geqslant}

\def\phi{\varphi}
\def\rho{\varrho}
\def\vartheta{\theta}

\newcommand{\numberset}{\mathbb}
\newcommand{\N}{\numberset{N}}
\newcommand{\R}{\numberset{R}}

\DeclareMathOperator*{\esssup}{ess\,sup}
\DeclareMathOperator*{\essinf}{ess\,inf}

\def\loc{{\rm loc}}
\newcommand{\px}{{p(\cdot)}}

\begin{document}

\title[Convergence of generalized Orlicz norms]{Convergence of generalized Orlicz norms with lower growth rate tending to infinity}

\date{\today}

\author{Giacomo Bertazzoni}
\address{Giacomo Bertazzoni, 
Dipartimento di Scienze Fisiche, Informatiche e Matematiche, 
Università degli Studi di Modena e Reggio Emilia, VIA CAMPI 213/B, 41125 Modena, Italy}
\email{\texttt{239975@studenti.unimore.it}}

\author{Petteri Harjulehto}
\address{Petteri Harjulehto,
Department of Mathematics and Statistics,
FI-00014 University of Helsinki, Finland}
\email{\texttt{petteri.harjulehto@helsinki.fi}}

\author{Peter Hästö}
\address{Peter Hästö, Department of Mathematics and Statistics,
FI-20014 University of Turku, Finland}
\email{\texttt{peter.hasto@utu.fi}}

\begin{abstract}
We study convergence of generalized Orlicz energies when the lower 
growth-rate tends to infinity. We generalize results by 
Bocea--Mih\u{a}ilescu (Orlicz case) and Eleuteri--Prinari (variable exponent case) 
and allow weaker assumptions: we are also able to handle unbounded domains with irregular boundary 
and non-doubling energies. 
\end{abstract}


\maketitle

%%%%%%%%%%%%%%%%%%%%%%%%%%%%%%%%%%%%%%%%%%%%%%%%%%%%%%%%%%%%%%%%%%%%%%%%
%%%%%%%%%%%%%%%%%%%%%%%%%%%%%%%%%%%%%%%%%%%%%%%%%%%%%%%%%%%%%%%%%%%%%%%%
%%%%%%%%%%%%%%%%%%%%%%%%%%%%%%%%%%%%%%%%%%%%%%%%%%%%%%%%%%%%%%%%%%%%%%%%
\section{Introduction}

%\marginpar{Possible journals
%\\Nonlinear Anal. (RWA), Bulicek, Marcellini
%\\Med. J. Math.
%\\Manus. Math., Bögelein, 
%\\Math. Nachr., Diening, Yang
%\\Rev. Mat. Complut.
%\\Proc. Edinb. Math. Soc.}
In this paper we consider $\Gamma$-convergence of a sequence of energies whose 
growth-rates tend to infinity. Let us first recap motivation and earlier studies. 
Garroni, Nesi and Ponsiglione \cite{GarNP01} studied dielectric breakdown using 
$\Gamma$-convergence of power-growth functions as the exponent tends to infinity. 
A more complicated functional of power-type was considered by Katzourakis \cite{Kat19} 
in a model for fluorescent optical tomography.
Prinari and co-authors \cite{ChaPP04, Pri15, PriZ20} considered an abstract version with 
the convergence of the energy 
\[
I_p(u) := \int_\Omega \tfrac1p f(x, u, Du)^p\, dx
\]
as $p\to \infty$. A similar problem without dependence on the derivative was considered 
in \cite{AnsP15, BocMP10, BocN08} and related to applications in polycrystal plasticity. 
If the energy is of non-standard growth-type \cite{Mar89}, 
then the $\Gamma$-limit can contain singular parts \cite{MinM05}. 

The energy $I_p$ has been generalized to two complementary cases: 
Orlicz growth by Bocea and Mih\u{a}ilescu \cite{BocM15} and variable exponent growth 
by Eleuteri and Prinari \cite{EleP21}. In this article, we consider 
the result in generalized Orlicz spaces, which covers both of these 
as special cases. Generalized Orlicz spaces (also known as Musielak--Orlicz spaces) 
and related PDE have been intensely studied recently, see, e.g., 
\cite{ChlGSW21, HarHL21, HasO22a, HurOS23, JiaWYYZ23} and references 
therein. Both the $L^\infty$- and the non-standard growth energy have been related to 
image processing \cite{CasMS98, CheLR06, HarH21}.

We follow the general approach of previous papers but we develop 
improved techniques with streamlined proofs and weaker assumptions. 
The tools used in \cite{BocM15, EleP21} led to the unexpected assumption that ratios of
the lower and upper growth-rates of the approximating functionals (denoted 
$\phi_n^\pm$ and $p_n^\pm$ in these papers) are bounded. 
We develop more precise estimates which allow us to avoid this assumption. 
Indeed, we need no assumptions about the upper growth-rate and can consider 
non-doubling functionals as well. 
For instance, we can handle the energy $|Du|^{n+1/|x|}$ that was excluded 
from previous results. In addition, we are able to deal with 
open sets of finite measure, whereas previous results required regularity of the boundary. 
We cover both norm- and modular-type energies with much the same method 
(cf.\ Theorems~\ref{thm:main-modular}). The general $\phi$-growth highlights the exact properties 
that are needed for the results more clearly than the $L^p$-case, 
as a comparison of the assumptions in 
Theorems~\ref{thm:main-norm} and \ref{thm:main-modular} shows, see also Example~\ref{eg:Einfty}.

We start by recalling some prior results related to Young measures and generalized 
Orlicz spaces. In Section~\ref{sect:embeddings}, we prove key tools for transferring 
results for constant exponents to the generalized Orlicz case. We conclude in 
Section~\ref{sect:main} with the main convergence results. 

%%%%%%%%%%%%%%%%%%%%%%%%%%%%%%%%%%%%%%%%%%%%%%%%%%%%%%%%%%%%%%%%%%%%%%%%
%%%%%%%%%%%%%%%%%%%%%%%%%%%%%%%%%%%%%%%%%%%%%%%%%%%%%%%%%%%%%%%%%%%%%%%%
%%%%%%%%%%%%%%%%%%%%%%%%%%%%%%%%%%%%%%%%%%%%%%%%%%%%%%%%%%%%%%%%%%%%%%%%
\section{Auxiliary results}
%In order to present the main theorem of this paper, we need some results, which will be presented in this section. 

%%%%%%%%%%%%%%%%%%%%%%%%%%%%%%%%%%%%%%%%%%%%%%%%%%%%%%%%
%\subsection*{Notation and terminology}

Throughout the paper we always consider a 
domain $\Omega \subset \R^N$, $N\ge 2$, i.e.\ an open and connected set. 
%By $p':=\frac p {p-1}$ we denote the H\"older conjugate exponent 
%of $p\in [1,\infty]$. The notation $f\lesssim g$ means that there exists a constant
%$c>0$ such that $f\le c g$. 
We denote the set of measurable functions $f: \Omega \to \R^d$ by $L^0(\Omega, \R^d)$. 
When $d=1$, the second set is omitted and we write $L^0(\Omega)$; the same convention is used 
for other function spaces.
The Lebesgue and Borel measures on $\R^k$ are denoted $\mathcal{L}^k$ and $\mathcal{B}_k$, 
respectively, and $\mathcal{L}^k(A)$ is abbreviated $|A|$.
%The notation $f\approx g$ means that
%$f\lesssim g\lesssim f$ whereas $f\simeq g$ means that 
%$f(t/c)\le g(t)\le f(ct)$ for some constant $c \ge 1$. 
%By $c$ we denote a generic constant whose
%value may change between appearances.
%A function $f$ is \textit{almost increasing} (more precisely, $L$-almost increasing) if there
%exists $L \ge 1$ such that $f(s) \le L f(t)$ for all $s \le t$.
%\textit{Almost decreasing} is defined analogously.
%By \textit{increasing} we mean that the inequality holds for $L=1$ 
%(some call this non-decreasing), similarly for \textit{decreasing}. 
%%%%%%%%%%%%%%%%%%%%%%%%%%%%%%%%%%%%%%%%%%%%%%%%%%%%%%%%%%%%%%%%%%%%%%%%
%\subsection*{Level convex functions}

\begin{dfn}
We say that $f: \R^N \to \R$ is \emph{level convex} if 
the level-set $\{ \xi \in \R^N: f(\xi) \le t\}$ is convex for every $t \in \R$.
\end{dfn} 

The following is a substitute for Jensen's inequality for level convex functions.

\begin{lem}[Theorem~1.2, \cite{BJW}]\label{lem:jensen}
Let $f: \R^N \to \R$ be a lower semicontinuous and level convex function and 
let $\mu$ be a probability measure in the open set $U \subset \R^N$. Then 
\[
f\bigg(\int_U u\,d\mu \bigg) \le 
\mu\text{-}\!\esssup_U f \circ u
\]
for every $u \in L^1_{\mu}(U, \R^N)$.
\end{lem}

%%%%%%%%%%%%%%%%%%%%%%%%%%%%%%%%%%%%%%%%%%%%%%%%%%%%%%%%%%%%%%%%%%%%%%%%
\subsection*{Young measures}

We recall some results regarding Young measures following the 
presentation of Section~2.3, \cite{EleP21} (which is based on \cite{Mu}). 

\begin{dfn}
A function $f: \Omega \times \R^d \times \R^k \to \R$ is called a \emph{normal integrand} if
\begin{itemize}
\item[$\bullet$] $f$ is $\mathcal{L}^N\times\mathcal{B}_d\times\mathcal{B}_k$-measurable;
\item[$\bullet$] $f(x,\cdot,\cdot)$ is lower semicontinuous for a.e. $x \in \Omega$.
\end{itemize}
\end{dfn}

We refer to Corollary~2.10 in \cite{EleP21} for the next result. 
The measures $\mu_x$ in the lemma are the so-called Young measures. 

\begin{lem}\label{lem:youngMeasure}
Suppose that $\Omega$ is bounded and $u_n\rightharpoonup u$ in $W^{1,q}(\Omega; \R^d)$, 
$q\in (1,\infty)$. 
Then there exists a family of probability measures $(\mu_x)_{x\in \Omega}$ on $\R^{Nd}$ 
such that 
\[ 
Du(x) = \int_{\R^{Nd}} \xi \,d\mu_x (\xi)\quad\text{for a.e. } x \in \Omega 
\]
and, for any normal integrand $f: \Omega \times \R^d \times \R^{Nd} \to \R^+$, 
\[
\liminf_{n \to \infty} \int_{\Omega} f(x, u_n(x), Du_n(x)) \,dx 
\ge \int_{\Omega} \int_{\R^{Nd}} f(x, u(x), \xi) \,d\mu_x(\xi)\,dx.
\]
\end{lem}

With the notation of the previous lemma and $v\in L^0(\Omega, \R^d)$, 
\cite[Remark~2.11]{EleP21} states that 
\[ 
\lim_{p \to \infty} \bigg( \int_{\Omega} \int_{\R^{Nd}} f(x, v(x), \xi)^p \,d\mu_x(\xi)\,dx\bigg)^{\frac{1}{p}} 
= 
\esssup_{x \in \Omega}\, \mu_x\text{-}\!\esssup_{\xi \in \R^k} f(x, v(x), \xi).
\]
Combined with the inequality 
$\mu_x\text{-}\!\esssup_{\xi \in \R^k} f(x, v(x), \xi) \ge f(x,v(x), Du(x))$
from Lemma~\ref{lem:jensen}, this implies the following:

\begin{cor}\label{cor:youngMeasure}
With the assumptions and notation of the previous lemma, we have 
\[
\lim_{p \to \infty} \bigg( \int_{\Omega} \int_{\R^{Nd}} f(x, v(x), \xi)^p \,d\mu_x(\xi)\,dx\bigg)^{\frac{1}{p}} 
\ge 
\esssup_{x \in \Omega} f(x, v(x), Du(x))
\]
provided that $f(x, z, \cdot)$ is level convex for every $z\in \R^d$ and a.e.\ $x\in \Omega$. 
\end{cor}

%%%%%%%%%%%%%%%%%%%%%%%%%%%%%%%%%%%%%%%%%%%%%%%%%%%%%%%%%%%%%%%%%%%%%%%%
\subsection*{Generalized Orlicz spaces}
In this subsection, we provide background on generalized Orlicz spaces. 
For more details, see \cite{ChlGSW21, HarH19}.

The next condition, which measures the lower growth-rate, 
can be stated as $\frac{f(t)}{t^p}$ being almost increasing (hence the abbreviation). 
We use an equivalent form which is easier to apply. 

\begin{dfn}\label{defn:aInc}\label{defn:aDec}
Let $f: \Omega\times [0,\infty) \to \R$ and $p > 0$. We say that $f$ satisfies 
%\begin{itemize}
%\item[(aInc)$_p$]
\ainc{p} if there exists $L\ge 1$ such that $f(x,\lambda t)\le L \lambda^p f(x,t)$ 
for a.e.\ $x\in \Omega$ and all $t\ge 0$, $\lambda \le 1$.
%\item[(aDec)$_q$]\label{defn:aDec}
%if there exists $L\ge 1$ such that $f(x,\lambda t)\le L \lambda^q f(x,t)$ 
%for all $t\ge 0$, $\lambda \ge 1$ and a.e.\ $x\in \Omega$.
%\end{itemize}
\end{dfn}

Note that if $\phi$ satisfies \ainc{p_1}, then it satisfies \ainc{p_2}, for every $p_2 < p_1$. 
The condition \adec{p}, related to upper growth-rate, is defined by the same inequality 
when $\lambda \ge 1$, but we do not need it in this paper. 

\begin{dfn}\label{def:Phi-function}
We say that $\phi: \Omega \times [0,\infty) \to [0,\infty]$ is a 
\emph{generalized weak $\Phi$-function} and write $\phi \in \Phiw(\Omega)$ if 
\begin{itemize}
\item
$x \mapsto \phi(x, |f(x)|)$ is measurable for every $f \in L^0(\Omega)$,
\item
$\phi(x, 0) = 0$, $\lim_{t \to 0^+} \phi(x, t) = 0$ and $\lim_{t \to \infty} \phi(x, t) = \infty$ 
for a.e. $x \in \Omega$, 
\item
$t \mapsto \phi(x, t)$ is increasing for a.e. $x \in \Omega$, 
\item
$\phi$ satisfies \ainc{1}.
\end{itemize}
If $\phi$ does not depend on $x$, then we omit the set and write $\phi \in \Phiw$.
\end{dfn}

We can now define generalized Orlicz spaces. Example~\ref{eg:Linfty} shows that this framework covers 
also $L^\infty$-spaces without the need for special cases. This is of special 
importance in this article as we consider the limit when the growth-rate tends to infinity. 

\begin{dfn}
Let $\phi \in \Phiw(\Omega)$ and let the modular $\rho_{\phi}$ be given by
\[ 
\rho_{\phi}(f) := \int_{\Omega} \phi(x, |f(x)|)\,dx
\]
for $f \in L^0(\Omega)$. The set
\[ 
L^{\phi}(\Omega) := \{ f \in L^0(\Omega) : \rho_{\phi}(\lambda f) < \infty \text{ for some } \lambda > 0  \}
\]
is called a \emph{generalized Orlicz space}. It is equipped with the Luxenburg quasinorm
\[ 
\|f\|_\phi 
:= 
\inf\Big\{ \lambda > 0 : \rho_{\phi}\Big(\frac{f}{\lambda}\Big) \le 1\Big\}.
\]
\end{dfn} 

%In the following, if the space and the measure are obvious from the context we will refer to $L^{\phi}(\Omega)$ as $L^{\phi}$.
%We abbreviate $\|f\|_{L^{\phi}(\Omega)} = \|f\|_{\phi}$ if the set and the measure are obvious from the context.
%In the following, we will often omit the prefix ``quasi'' but $\|\cdot\|_{\phi}$ turns out to be an actual norm if $\phi(x,\cdot)$ is a left-continuous, convex function, for a.e. $x \in \Omega$.

\begin{example}\label{eg:Linfty}
Define $\phi_\infty\in \Phiw(\Omega)$ by $\phi_\infty(x, t):=\infty \chi_{(1,\infty)}(t)$. 
Then $\phi_\infty$ is a generalized weak $\Phi$-function. From the definition 
of modular we see that 
\[
\rho_{\phi_\infty}(f) = 
\begin{cases}
0, &\text{if } |f|\le 1 \text{ a.e.} \\
\infty, &\text{otherwise}.
\end{cases}
\]
It follows from the Luxemburg norming procedure that $\|\cdot\|_{\phi_\infty}=\|\cdot\|_\infty$
and so $L^{\phi_\infty}(\Omega)=L^\infty(\Omega)$. 
\end{example}

Other examples of generalized Orlicz spaces are (ordinary) Orlicz spaces where $\phi(x,t)$ is independent of $x$, variable exponent spaces $L^{\px}$ where $\phi(x, t) = t^{p(x)}$ \cite{DieHHR11},  double phase spaces $\phi(x, t) = t^p + a(x) t^q$ \cite{BarCM18, GasP23}, variable exponent double phase spaces $\phi(x, t)= t^{p(x)} + a(x) t^{q(x)}$  \cite{CBGHW22} and many other variants (e.g.\ \cite{BaaB22}). 
The Orlicz--Sobolev space is defined based on $L^\phi(\Omega)$ as usual: 

\begin{dfn}
Let $\phi \in \Phiw(\Omega)$. The function $u \in L^{\phi}(\Omega) \cap W^{1,1}_\loc(\Omega)$ belongs to the \emph{Orlicz--Sobolev space $W^{1,\phi}(\Omega)$}, if the weak partial derivatives $\frac{\partial u}{\partial x_i}$, $i = 1, \dots, N$, belong to $L^{\phi}(\Omega)$. We define a modular and quasinorm on $W^{1, \phi}(\Omega)$ by
\[ 
\rho_{1, \phi}(u): = \rho_\phi(u) + \sum_{i = 1}^N \rho_\phi\bigg(\frac{\partial u}{\partial x_i}\bigg) 
\quad\text{and}\quad
\|u\|_{1,\phi} := \inf\Big\{ \lambda > 0 : \rho_{1,\phi} \Big(\frac{u}{\lambda}\Big) \le 1\Big\}.
\]
\end{dfn}
%In the following, when there is ni danger of confusion, we will refer to $\rho_{W^{1, \phi}(\Omega)}$ and  $\|u\|_{W^{1,\phi}(\Omega)}$ as $\rho_{1, \phi}$ and  $\|u\|_{1,\phi}$.

%%%%%%%%%%%%%%%%%%%%%%%%%%%%%%%%%%%%%%%%%%%%%%%%%%%%%%%%%%%%%%%%%%%%%%%%
%%%%%%%%%%%%%%%%%%%%%%%%%%%%%%%%%%%%%%%%%%%%%%%%%%%%%%%%%%%%%%%%%%%%%%%%
%%%%%%%%%%%%%%%%%%%%%%%%%%%%%%%%%%%%%%%%%%%%%%%%%%%%%%%%%%%%%%%%%%%%%%%%
\section{Asymptotically sharp embeddings}\label{sect:embeddings}

The unit ball property is a fundamental relation between quasinorm and modular. 
Note that we do not assume lower semi-continuity of $\phi$ which complicates 
the relationship is a bit. 

\begin{lem}[Unit ball property, Lemma 3.2.3, \cite{HarH19}] \label{lemma:Unit-ball}
If $\phi \in \Phiw(\Omega)$, then
\[ 
\|f\|_{\phi} < 1 
\quad\Rightarrow\quad 
\rho_{\phi}(f) \le 1 
\quad\Rightarrow\quad 
\|f\|_{\phi} \le 1.
\]
\end{lem}

This next property anchors $\phi$ at $1$. 
It follows from \ainc{1} that if $\frac1c \le \phi(x, 1) \le c$ for a.e.\ $x \in \Omega$, 
then $\phi$ satisfies \azero{} with $\beta := \frac1{Lc}$.

\begin{dfn}\label{def:a0}
We say that a  function $\phi \in \Phiw(\Omega)$ satisfies \azero{} if there exists $\beta \in (0,1]$ such that $\phi(x,\beta) \le 1 \le \phi(x, \frac{1}{\beta})$ for a.e.\ $x \in \Omega$.
\end{dfn}

The following proposition plays a significant role in reducing our main problem to 
the classic Lebesgue space $L^q(\Omega)$, where we can use the results about Young measures previously presented.
The embedding is well-known (see \cite[Lemma 3.7.7]{HarH19}) but we need a version 
with an asymptotically sharp constant as $p\to\infty$.

\begin{prp}\label{prop:embedding}
Let $\Omega$ have finite measure, $\phi\in \Phiw(\Omega)$ and $c, L\ge 1$. Assume  that 
\begin{itemize}
\item $\frac1c \le \phi(x,1)\le c$ 
for a.e.\ $x \in \Omega$;
\item $\phi$ satisfies \ainc{p} with constant $L$ for some $p\ge 1$.
\end{itemize}
Then $L^{\phi}(\Omega) \hookrightarrow L^p(\Omega)$ and 
\[  
\|\cdot\|_{L^p} \le (2L(|\Omega|+c))^\frac1p \|\cdot\|_{\phi}.
\]
\end{prp}
\begin{proof}
The condition \azero{} follows from $\frac1c \le \phi(x,1) \le c$ and \ainc{p}, and so the embedding holds \cite[Lemma~3.7.7]{HarH19}. By the unit ball property (Lemma~\ref{lemma:Unit-ball}), the 
norm inequality $\|u\|_{L^p} \le C \|u\|_{\phi}$ follows once we prove that 
\[
\rho_\phi\bigg(\frac{C\,|u|}{\|u\|_p}\bigg) = \int_\Omega \phi\bigg(x,\frac{C\,|u|}{\|u\|_p}\bigg)\, dx > 1.
\]
Let us derive a lower bound for $\phi$. By \ainc{p}, 
\[
\phi(x, t) \ge \frac1L t^p \phi(x,1) \ge \frac1{Lc} t^p
\]
when $t\ge L^{1/p}\ge 1$. If we subtract $\frac1c$ from the right-hand side, we obtain a function which 
is negative when $t< L^{1/p}$ and therefore a suitable lower bound also when $t$ is in this range. 
It follows that 
\[
\phi(x, t) \ge  \frac1{Lc} t^p - \frac1c
\]
for every $t\ge 0$. Applying this in the modular, we see that 
\[
\int_\Omega \phi\bigg(x,\frac{C\,|u|}{\|u\|_p}\bigg)\, dx
\ge
\int_\Omega \frac1{Lc}\bigg(\frac{C\,|u|}{\|u\|_p}\bigg)^p - \frac1c\, dx
=
\frac{C^p}{Lc} - \frac{|\Omega|}{c}.
\]
Thus the desired norm-inequality holds if the right-hand side is greater than $1$, 
which is equivalent to $C^p > L(|\Omega|+c)$.
\end{proof}

The next proposition essentially takes care of the $\limsup$-part of the estimate in 
the main results. The almost increasing assumption is satisfied (with $L=1$) 
for example when $\phi_n^{1/p_n}$ is convex.

\begin{prp}\label{prop}
Let $\Omega$ have finite measure, 
$\phi_n\in \Phiw(\Omega)$ for $n\in\N$ and $c, L\ge 1$. Assume for $n\in\N$ that 
\begin{itemize}
\item $\frac1c \le \phi_n(x,1) \le c$ for a.e.\ $x \in \Omega$;
\item $\phi_n$ satisfies \ainc{p_n} with constant $L$ for some $p_n\ge 1$.
%\item $p_n \to \infty$ as $n \to \infty$.
\end{itemize}
If $p_n\to \infty$, then $\lim_{n \to \infty} \|u\|_{\phi_n} = \|u\|_{\infty}$ for all $u \in L^{\infty}(\Omega)$.
\end{prp}

\begin{proof}
If $\|u\|_\infty=0$, then $u=0$ and there is nothing to prove. 
It suffices to consider the case $\|u\|_\infty=1$, since the general case can be reduced to it by 
the normalization $\tilde u := \frac u{\|u\|_\infty}$. 

By the definition of limit, we have to find for every $\epsilon \in (0,1)$ a number $n_0 \in \N$ such that  
\[ 
\frac1{1+\epsilon} \le \|u\|_{\phi_n} \le \frac1{1-\epsilon}
\]
for all $n \ge n_0$.
By the unit ball property (Lemma~\ref{lemma:Unit-ball}) this inequality follows from
\[ 
\rho_{\phi_n}((1-\epsilon)u) \le 1 <\rho_{\phi_n}((1+\epsilon)u). 
\]

Fix $\epsilon \in (0,1)$. 
By \ainc{p_n}, $\phi_n(x, 1 -\epsilon) \le L(1 -\epsilon)^{p_n}\phi_n(x, 1)
\le Lc(1 -\epsilon)^{p_n}$. 
Since $|u| \le \|u\|_\infty = 1$ a.e.\ and $\phi_n$ is increasing, 
we conclude that
\[ 
\rho_{\phi_n}((1-\epsilon)u)=\int_{\Omega}\phi_n(x, (1 - \epsilon) |u| )\, dx 
\le \int_{\Omega}\phi_n(x, 1 -\epsilon)\, dx
\le Lc(1 -\epsilon)^{p_n} |\Omega| \overset{n \to \infty}{\longrightarrow} 0.
\]
Thus we can find $n_0$ such that $\rho_{\phi_n}((1-\epsilon)u)<1$ for all $n \ge n_0$. 

To deal with the second inequality, we consider $A:=\{(1+\epsilon)|u| > 1+\frac\epsilon2\}$. 
Since $\|u\|_\infty=1$, it follows that $|A| > 0$. By \ainc{p_n}, 
\[
\phi_n(x, 1+\tfrac\epsilon2)\ge \tfrac1L(1+\tfrac\epsilon2)^{p_n}\phi_n(x, 1)
\ge \tfrac1{Lc}(1+\tfrac\epsilon2)^{p_n}.
\] 
The inequality $1 <\rho_{\phi_n}((1+\epsilon)u)$ follows for large $n$ from the estimate
\[ 
\int_{\Omega} \phi_n(x, (1+\epsilon)|u|)\,dx 
\ge \int_A \phi_n(x, 1+\tfrac\epsilon2)\,dx
\ge
\tfrac1{Lc}(1+\tfrac\epsilon2)^{p_n} |A|
\overset{n \to \infty }{\longrightarrow } \infty.  \qedhere
\]
\end{proof}

The previous two results used the assumption $\frac1c \le \phi_n(x,1) \le c$ which is 
stronger that \azero{}. This is because \azero{} is not sufficient, as the next example demonstrates.

\begin{example}
Define $\phi_n\in \Phiw$ by $\phi_n(t):=(a_n t)^n$, where $\beta \le a_n\le \frac1\beta$. 
Then $\phi_n$ satisfies \azero{} with constant $\beta$. If $a_n\to a \ne 1$, then 
$\|\cdot\|_{\phi_n} \to \|\cdot\|_{\phi_{a,\infty}} = a \|\cdot\|_\infty$, where 
$\phi_{a,\infty}(t) := \infty \chi_{(1/a, \infty)}(t)$. Therefore, the 
conclusion of the previous proposition does not hold in this case. 
If $a>1$, then the inequality $\|u\|_{\phi_n}\le c_n\|u\|_\infty$
of Proposition~\ref{prop:embedding} similarly does not 
hold with any constant $c_n\to 1$. 
\end{example}

The second example shows that 
the assumption that each $\phi_n$ satisfy \ainc{p_n} with the same constant $L$ is also needed 
for the conclusion.

\begin{example}
Define $\phi\in \Phiw$ by $\phi(t):= \max\{0, 2t-1\} + \phi_\infty(t)$.
We consider the constant sequence $\phi_n=\phi$, whose Luxemburg quasinorms 
converge to $\|\cdot\|_{\phi}$, not $\|\cdot\|_\infty$. Since $\phi_n(1)=1$, the first condition 
of Proposition~\ref{prop} is satisfied. 
Let $(p_n)$ be a sequence converging to $\infty$.
We show that $\phi_n$ satisfies  \ainc{p_n} with a constant $L_n\ge 1$. % for some $p_n\to\infty$. 
If $t \le \frac12$, then $\phi(t) =0$, and if $t >1$, then $\phi(t) = \infty$.
Thus it is enough to show the 
second inequality in the following when $\frac12 \le \lambda t \le t \le 1$:
\[
\phi_n(\lambda t) = 2\lambda t - 1 \le (2t-1) \lambda 
\le 
L_n\lambda^{p_n} (2t-1) = L_n\lambda^{p_n} \phi_n(t).
\]
This inequality holds when $L_n \ge \lambda^{1-p_n}$. Since $\lambda\ge\frac12$, 
the choice $L_n := 2^{p_n-1}$ works for \ainc{p_n}. 
\end{example}

Note that both examples are of Orlicz type as the dependence of $\phi$ on $x$ is not used here.

%%%%%%%%%%%%%%%%%%%%%%%%%%%%%%%%%%%%%%%%%%%%%%%%%%%%%%%%%%%%%%%%%
%%%%%%%%%%%%%%%%%%%%%%%%%%%%%%%%%%%%%%%%%%%%%%%%%%%%%%%%%%%%%%%%%
%%%%%%%%%%%%%%%%%%%%%%%%%%%%%%%%%%%%%%%%%%%%%%%%%%%%%%%%%%%%%%%%%
\section{Main theorems}\label{sect:main}

Let $\phi\in\Phiw(\Omega)$ and $f: \Omega \times \R^d \times \R^{Nd} \to \R$. 
We define $E_\phi, F_\phi : L^1(\Omega, \R^d) \to [0,\infty]$  by 
\[
E_\phi(u) := \begin{cases} 
\rho_\phi(f(\cdot, u, Du)) &\text{if } u \in W^{1,1}_\loc(\Omega, \R^d), \\ 
\infty &\text{otherwise},
\end{cases}
\]
and
\[
F_\phi(u) := \begin{cases} 
\|f(\cdot, u, Du)\|_\phi &\text{if } u \in W^{1,1}_\loc(\Omega, \R^d), \\ 
\infty &\text{otherwise}.
\end{cases}
\]
The abbreviations $E_\infty$ and $F_\infty$ refer to the case
$\phi=\phi_\infty$ from Example~\ref{eg:Linfty}, related $L^\infty$. 

\begin{rem}
It follows from Theorems~\ref{thm:main-norm} and \ref{thm:main-modular} 
that $(F_{\phi_n})$ and $(E_{\phi_n})$ $\Gamma$-converge to $F_\infty$ and $E_\infty$ respectively, with respect to the $L^1$-weak topology. Specifically, we prove the 
$\liminf$-property of $\Gamma$-convergence and show that we can use a constant recovery 
sequence; we refer the reader to \cite{Bra02} for the definitions related to $\Gamma$-convergence
which are not otherwise needed here. 
\end{rem}
%Next we study modular-type operators.
%Let $(\phi_n)$ be a sequence of weak $\Phi$-functions so that each $\phi_n$
% satisfies \ainc{p_n}.
%We define  by 
%\[
%E_n(u) := \begin{cases} 
%\frac{1}{p_n} \int_\Omega \phi_n (x, f(\cdot, u, Du)) \, dx , &\text{ if } u \in W^{1,1}_\loc(\Omega, \R^d); \\ 
%\infty, &\text{otherwise}.
%\end{cases}
%\]
%and
%$E_\infty: L^1(\Omega, \R^d) \to [0,\infty]$ by 
%\[
%E_\infty(u) := \begin{cases} 
%0, &\text{ if } u \in W^{1,1}_\loc(\Omega, \R^d) \text{ and } \|f(\cdot, u, Du)\|_\infty \le 1; \\ 
%\infty, &\text{otherwise}.
%\end{cases}
%\]

\begin{thm}\label{thm:main-norm}
Assume that $\Omega$ has finite measure. Let 
$f: \Omega \times \R^d \times \R^{Nd} \to \R$ be a normal integrand such that 
\begin{itemize}
\item
$f(x,u,\cdot)$ is level convex for a.e.\ $x \in \Omega$ and every $u \in \R^d$;
\item
$f(x, u, \xi) \ge \alpha|\xi|^{\gamma}$ 
for some $\alpha, \gamma > 0$, a.e.\ $x \in \Omega$ and every $(u,\xi) \in \R^d \times \R^{Nd}$. 
\end{itemize}
Let $\phi_n\in \Phiw(\Omega)$ for $n\in\N$ and $c, L\ge 1$ and assume that 
\begin{itemize}
\item $\frac1c \le \phi_n(x,1) \le c$ for a.e.\ $x \in \Omega$;
\item $\phi_n$ satisfies \ainc{p_n} with constant $L$ for some $p_n\ge 1$.
\end{itemize}
If $p_n \to \infty$ as $n \to \infty$, then
\[ 
\limsup _{n \to \infty} F_{\phi_n}(u) \le  F_\infty(u) 
\le \liminf _{n \to \infty} F_{\phi_n}(u_n)
\]
for all $u, u_n \in L^1(\Omega, \R^d)$ with $u_n \rightharpoonup u$ in $L^1(\Omega, \R^d)$.
\end{thm}

\begin{proof}
For simplicity we abbreviate $F_n:= F_{\phi_n}$, $n\in\N$. 
Let us first show that
\[
\limsup_{n \to \infty} F_n(u) \le F_\infty(u)
\]
for $u \in L^1(\Omega, \R^d)$. If $F_\infty(u) = \infty$, this is clear.
Otherwise, $F_\infty(u) < \infty$ so $f(\cdot, u, Du) \in L^{\infty}(\Omega)$ and 
Proposition~\ref{prop} gives that 
\[ 
\lim_{n \to \infty} F_n(u) = \lim_{n \to \infty} \|f(\cdot, u, Du)\|_{\phi_n} = \|f(\cdot, u, Du)\|_{\infty} = F_\infty(u).
\]

We next deal with the $\liminf$-inequality. Let $(u_n)_n \subset L^1(\Omega, \R^d)$ 
converge weakly to $u \in L^1(\Omega, \R^d)$. Without loss of generality, we assume that 
\[ 
\liminf_{n \to \infty} F_n(u_n) = \lim_{n \to \infty} F_n(u_n) = M < \infty.
\]
Then every subsequence of $(u_n)_n$ also has limit $M$. 
Recall that $p_n\to\infty$. Fix $q>\gamma$ and let $n_0 \in \N$ be such that
\[
p_n \ge \frac q\gamma \quad\text{and}\quad F_n(u_n) \le M + 1\qquad\text{for all } n \ge n_0.
\]
From $p_n \ge \frac q\gamma$ it follows that $\phi_n$ satisfies \ainc{q/\gamma} with 
the same constant $L$ as in the assumption. By Proposition~\ref{prop:embedding},
\begin{equation}\label{eq:norm-inequality}
\|f(\cdot, u_n, Du_n)\|_{q/\gamma}  
\le  
\underbrace{\big(2L(|\Omega|+c)\big)^\frac\gamma q}_{=:C_q} 
\|f(\cdot, u_n, Du_n)\|_{\phi_n} 
\le
C_1 (M+1).
\end{equation}
for every $n \ge n_0$.
We note that $C_q\searrow 1$ as $q\to \infty$.

We use this inequality and the assumed lower bound on $f$ to estimate the norm of the gradient:
\[
\|Du_n\|^{\gamma}_{q} 
= 
\big\||Du_n|^\gamma\big\|_{q/\gamma} 
\le \tfrac{1}{\alpha}\|f(\cdot, u_n, Du_n)\|_{q/\gamma}
\le \tfrac{C_1}{\alpha}  (M+1).
\]
It follows that $(\|Du_n\|_q )_n$ is bounded.
Then, up to a subsequence (depending on $q$), $(Du_n)_n$ weakly converges to a function $w$ in $L^q(\Omega, \R^{Nd})$. 
Since $(u_n)_n$ weakly converges to $u$ in $L^1(\Omega, \R^d)$, $w$ is the distributional gradient of $u$. 
%In particular, $u \in W^{1,1}(\Omega, \R^d)$.  
%
Fix an open ball $B\subset\Omega$. 
We use a version of the Poincar\'e inequality where $L^q$-integrability is required only for the gradient. By the theorem in Section 1.5.2 \cite[p. 35]{MazP97}, 
we obtain that 
%\begin{align*}
%\|u_n\|_{q} 
%\le 
%\|u_n- (u_n)_B\|_{q}  + \|(u_n)_B\|_{q} \le 
%c(N, B) \|Du_n\|_{q} + |B|^{\frac1q-1} \|u_n\|_{1}.
%\end{align*}
\begin{align*}
\|u_n\|_{L^q(B, \R^d)} 
&\le 
\|u_n- (u_n)_B\|_{L^q(B, \R^d)}  + \|(u_n)_B\|_{L^q(B, \R^d)} \\
&\le 
c(N, B) \|Du_n\|_{L^q(B, \R^d)} + |B|^{\frac1q-1} \|u_n\|_{L^1(B, \R^d)},
\end{align*}
Since $(\|Du_n\|_q )_n$ is bounded and $u_n \rightharpoonup u$ in $L^1(\Omega, \R^{d})$, the right hand side is uniformly bounded.
Thus reflexivity in $L^q(B, \R^{d})$ yields that $(u_n)_n$ has a weakly convergent subsequence.

We have found a subsequence, denoted still by $(u_n)_n$, with 
$u_n \rightharpoonup u$ in $W^{1,q}(B, \R^{d})$. 
Let $(\mu_x)_{x \in B}$ be the family of probability measures 
(the ``Young measure'') from Lemma~\ref{lem:youngMeasure} corresponding to this subsequence.
We use the first inequality from \eqref{eq:norm-inequality} in $B$ and 
Lemma~\ref{lem:youngMeasure} for the normal integrand $f^{q/\gamma}$ 
to conclude that 
\begin{align*}
\liminf_{n\to\infty} F_n(u_n) 
&\ge 
\liminf_{n\to\infty} \frac1{C_q} \bigg(\int_B f(x, u_n(x), Du_n(x))^{\frac q\gamma} \, dx\bigg)^{\frac{\gamma}{q}}\\
&\ge 
\frac1{C_q} \bigg( \int_B \int_{\R^{Nd}} f(x, u(x), \xi)^{\frac q\gamma} \, d\mu_x(\xi) \, dx\bigg)^{\frac{\gamma}{q}}.
\end{align*}
Take the limit inferior on the right-hand side as $q\to\infty$ and use 
$C_q\to 1$ as well as Corollary~\ref{cor:youngMeasure}:
\begin{align*}
\liminf_{n \to \infty} F_n(u_n) &
\ge \liminf_{q \to \infty}\bigg( \int_B \int_{\R^{Nd}} f(x, u(x), \xi)^{\frac q\gamma}\, d\mu_x(\xi)\, dx\bigg)^{\frac{\gamma}{q}}\\
&\ge \esssup_{x \in B} f(x, u(x), Du(x))
=
\sup_{x \in B\setminus A} f(x, u(x), Du(x)), 
\end{align*}
%Since $f(x, u(x), \cdot)$ is level convex for a.e. $x \in B$, taking into account \eqref{young1}, by Lemma~\ref{lem:jensen} we have that
%\[ 
%\mu_x\text{-}\!\esssup_{\xi \in \R^{Nd}} f(x, u(x), \xi)
%\ge
%f\bigg( x, u(x), \int_{\R^{Nd}} \xi\, d\mu_x(\xi)\bigg) 
%=
%f(x, u(x), Du(x))  
%\]
%for a.e. $x \in B$. Taking the essential supremum over both sides, we continue our previous 
%estimate by 
%\[ 
%\liminf_{n \to \infty} F_n(u_n) 
%\ge 
%\esssup_{x \in B}\, \mu_x\text{-}\!\esssup_{\xi \in \R^{Nd}} f(x, u(x), \xi)
%\ge
%\sup_{x \in B\setminus A} f(x, u(x), Du(x)),
%%=F_\infty(u). 
%\]
where $A\subset B$ is the exceptional set of measure zero related to the 
essential supremum. We can write $\Omega=\cup_{i\in \N} B_i$ as a countable union of balls $B_i$ 
and then we obtain an estimate for the supremum in $\Omega \setminus \cup_{i\in \N} A_i \subset 
\cup_{i\in \N} (B_i \setminus A_i)$, where $A_i$ is the exceptional set 
corresponding to $B_i$. Since also $\cup_{i\in \N} A_i$ has measure zero, 
this gives the essential supremum in $\Omega$. Thus we have proved the $\liminf$-inequality and 
completed the proof.  
\end{proof}

We conclude with a similar result for the modular-based energy functional. 
We use the notation $\displaystyle\phi^-(1):= \essinf_{x \in \Omega} \phi(x, 1)$ and
$\displaystyle\phi^+(1):= \esssup_{x \in \Omega} \phi(x, 1)$.

\begin{thm}\label{thm:main-modular}
Assume that $\Omega$ has finite measure. Let 
$f: \Omega \times \R^d \times \R^{Nd} \to \R$ be a normal integrand such that
\begin{itemize}
\item
$f(x,u,\cdot)$ is level convex for a.e.\ $x \in \Omega$ and every $u \in \R^d$;
\item
$f(x, u, \xi) \ge \alpha|\xi|^{\gamma}$ 
for some $\alpha, \gamma > 0$, a.e.\ $x \in \Omega$ and every $(u,\xi) \in \R^d \times \R^{Nd}$. 
\end{itemize}
Let $\phi_n\in \Phiw(\Omega)$ for $n\in\N$ and $L\ge 1$ and assume that 
\begin{itemize}
\item $\limsup_{n \to \infty} \phi_n^+(1)= 0$ and $\liminf_{n \to \infty} \phi_n^-(1)^{1/p_n}\ge 1$; 
\item $\phi_n$ satisfies \ainc{p_n} with  constant $L$ for some $p_n\ge 1$.
\end{itemize}
If $p_n \to \infty$ as $n \to \infty$, then
\[ 
\limsup _{n \to \infty} E_{n}(u) \le  E_\infty(u) 
\le \liminf _{n \to \infty} E_{n}(u_n)
\]
for all $u, u_n \in L^1(\Omega, \R^d)$ with $u_n \rightharpoonup u$ in $L^1(\Omega, \R^d)$.
\end{thm}

\begin{proof}
We again abbreviate $E_n:= E_{\phi_n}$, $n\in\N$. 
Let us first show that
\[
\limsup_{n \to \infty} E_n(u) \le E_\infty(u)
\]
for any $u \in L^1(\Omega, \R^d)$. If $E_\infty(u) = \infty$, this is clear.
Otherwise, $E_\infty(u) =0$ so $\|f(\cdot, u, Du)\|_\infty \le 1$. Thus we obtain
that
\[ 
\begin{split}
\limsup_{n \to \infty} E_n(u) &= 
\limsup_{n \to \infty} \int_\Omega \phi_n (x, f(\cdot, u, Du)) \, dx
\le \limsup_{n \to \infty} |\Omega|\, \phi_n^+ (1) =0 
\end{split}
\]
and the inequality follows. 

To deal with the $\liminf$-inequality, let $(u_n)_n \subset L^1(\Omega, \R^d)$ 
converge weakly to $u \in L^1(\Omega, \R^d)$. Without loss of generality, we assume that 
\[ 
\liminf_{n \to \infty} E_n(u_n) = \lim_{n \to \infty} E_n(u_n) = M < \infty.
\]
Define $\hat\phi_n\in \Phiw(\Omega)$ by $\hat\phi_n(x, t):=\frac{\phi_n(x, t)}{\phi_n(x, 1)}$.
Let $n_0 \in \N$ be such that $E_n(u_n) \le M+1$ for all $n\ge n_0$, so that 
\[
\int_\Omega \hat\phi_n(x, f(x, u_n, Du_n)) \, dx \le \frac{M+1}{\phi_n^-(1)}.
\]
Since $\hat\phi_n$ satisfies \ainc{p_n} with the same constant $L$, this yields
by Corollary 3.2.10 of \cite{HarH19} that
\[
\|f(\cdot, u_n, Du_n)\|_{\hat\phi_n} 
\le
\max\Big\{ \big( L \rho_{\hat\phi_n}(f(\cdot, u_n, Du_n)) \big)^{\frac1{p_n}}, 1 \Big\}
\le 
\max\bigg\{ \Big(\frac{L(M+1)}{\phi_n^-(1)}\Big)^{\frac1{p_n}}, 1 \bigg\},
\]
and consequently 
\[
\limsup_{n \to \infty} \|f(\cdot, u_n, Du_n)\|_{\hat\phi_n} 
\le 
\max\Big\{ \limsup_{n \to \infty} \phi_n^-(1)^{-\frac1{p_n}}, 1 \Big\}
=
1.
\]
Since $\hat\phi_n$ satisfies the requirements of Theorem~\ref{thm:main-norm} with the same 
$L$ and $c=1$, it follows that 
$F_\infty(u) \le 1$, i.e.\ $u \in W^{1,1}_\loc (\Omega, \R^d)$ and  $\|f(\cdot, u, Du)\|_{\infty} \le 1$.
Thus $E_\infty (u)\le \rho_\infty(1) =0$, and hence the $\liminf$-inequality holds.
\end{proof}

The assumptions of Theorem~\ref{thm:main-modular} differ from Theorem~\ref{thm:main-norm}. 
The next example shows that the condition $\frac1c \le \phi_n(x,1) \le c$ is not sufficient 
in the latter theorem. 

\begin{example}\label{eg:Einfty}
Let $p_n:=n$ for all $n$. Then $\phi_n(t):= t^n$ satisfies \ainc{n} and 
$\frac1c \le \phi_n(x,1) \le c$ but not the condition 
$\limsup_{n\to\infty}\phi_n^+(1) = 0$ of the previous theorem. 
Moreover, the claim of the theorem 
also does not hold: if $f$ and $u$ are such that $f(x, u, Du)=1$ a.e., then 
$E_n(u) = |\Omega|$ for every $n\in \N$ but  $E_\infty(u)=0$ so that 
\[
\limsup _{n \to \infty} E_{n}(u) >  E_\infty(u). 
\]
On the other hand, the energy $\phi_n(t):= \frac1n t^n$ does satisfy the 
conditions the previous theorem. 
%
%
%The conditions $\phi_n^+(1)\to 0$ and $\liminf_n \phi_n^-(1)^{1/p_n}\ge 1$ are satisfied 
%by $\phi_n(t):= \frac1n t^n$ but not by $\phi_n(t):= t^n$. 
%Both weak $\Phi$-functions 
%However, the claim of the theorem 
%also does not hold in the second example: if $f$ and $u$ are such that $f(x, u, Du)=1$ a.e., then 
%$E_n(u) = |\Omega|$ for every $n\in \N$ but  $E_\infty(u)=0$ so that 
%\[
%\limsup _{n \to \infty} E_{n}(u) >  E_\infty(u). 
%\]
\end{example}

%%%%%%%%%%%%%%%%%%%%%%%%%%%%%%%%%%%%%%%%%%%%%%%%%%%%%%%%%%%%%%%%%%%%%%%%
%%%%%%%%%%%%%%%%%%%%%%%%%%%%%%%%%%%%%%%%%%%%%%%%%%%%%%%%%%%%%%%%%%%%%%%%
%%%%%%%%%%%%%%%%%%%%%%%%%%%%%%%%%%%%%%%%%%%%%%%%%%%%%%%%%%%%%%%%%%%%%%%%

\end{document}